\documentclass[letterpaper, 10 pt, conference]{ieeeconf}

\IEEEoverridecommandlockouts                              

\title{A Compositional Approach to Certifying the Almost \\ Global Asymptotic Stability of Cascade Systems}
\author{Jake Welde, Matthew D. Kvalheim, and Vijay Kumar%
\thanks{
	J. Welde and V. Kumar are with the GRASP Laboratory at the University of Pennsylvania, while M. D. Kvalheim is with the Department of Mathematics at the University of Michigan. emails: \texttt{\{jwelde,kumar\}@seas.upenn.edu}, \texttt{kvalheim@umich.edu}.  
	We gratefully acknowledge the support of Qualcomm Research, NSF Grant CCR-2112665, and the NSF Graduate Research Fellowship Program.
}
}

\date{\today}

\usepackage{amsthm}
\usepackage{amsmath}
\usepackage{tikz}
\usetikzlibrary{shapes,arrows}

\usepackage{amssymb}
\usepackage{thmtools}
\usepackage{graphicx}
\usepackage{mathtools}
\let\labelindent\relax
\usepackage[shortlabels]{enumitem}

\declaretheoremstyle[notefont=\normalfont\itshape,bodyfont=\normalfont]{normaltext}

\newtheorem{theorem}{Theorem}
\newtheorem{fact}{Fact}

\newtheorem{corollary}{Corollary}
\declaretheorem[name=Definition,style=normaltext]{definition}
\declaretheorem[name=Remark,style=normaltext]{remark}

\newcommand{\R}{\mathbb{R}}

\declaretheoremstyle[notefont=\normalfont\itshape,bodyfont=\normalfont\itshape,headfont=\normalfont\sc]{smallcapsname}
\declaretheorem[name=Step,style=smallcapsname]{step}
\newtheorem{proposition}{Proposition}

\DeclareMathOperator{\dist}{dist}

\usepackage{cite}

\usepackage{balance}

\renewcommand\qedsymbol{$\blacksquare$}

\begin{document}

\makeatletter
\@addtoreset{step}{theorem}
\makeatother

\maketitle

\begin{abstract}
In this work, we give sufficient conditions for the almost global asymptotic stability of a cascade in which the subsystems are only almost globally asymptotically stable. The result is extended to upper triangular systems of arbitrary size. In particular, if the unforced subsystems are almost globally asymptotically stable and their only chain recurrent points are hyperbolic equilibria, then the 
boundedness of forward trajectories is sufficient for the almost global asymptotic stability of the full upper triangular system.
We show that unboundedness of such cascades is prohibited by growth rate conditions on the interconnection term and a Lyapunov function for the unforced outer subsystem, and the required structure for the chain recurrent set is enjoyed by classes of systems common in geometric control e.g. dissipative mechanical systems.
Our results stand in contrast to prior works that require either time scale separation, prohibitively strong disturbance robustness properties, or \textit{global} asymptotic stability in the subsystems. 
\end{abstract}

\vspace{-2pt}
\section{Introduction}
\vspace{-2pt}

In this work, we are interested in 
cascades of the form
\begin{subequations}
	\begin{align}
		\dot{x} &= f(x,y), \label{outer_general} 
		\\ 
		\dot{y} &= g(y), \label{inner_general}
	\end{align}
\end{subequations}
depicted graphically in in Fig. 1 as  system $\Sigma$. We assume that $x$ and $y$ evolve on $X$ and $Y$, which are smooth, connected, complete Riemannian manifolds without boundary (see Remark \ref{choice_of_setting} for an explanation of the choice of setting).

Cascades appear in many interesting and important physical systems. Many underactuated mechanical systems can be rendered as a cascade after a feedback transformation \cite{Olfati2000}, and cascades arise often in robotic systems, either intrinsically 
or after control design \cite{Lee2010}. A long research tradition has studied the implications of cascade structure to simplify analysis and aid design
\cite{Kokotovic2001}%
. This compositional approach is motivated by the observation that control design for a subsystem is typically easier, due to e.g. lower dimensionality, lower relative degree, or full actuation. 
However, one must ensure that the recombined system achieves the desired behavior, since only \textit{local}  asymptotic stability is preserved under cascades for general nonlinear systems \cite{Sepulchre1997}.
	\vspace{-14pt}
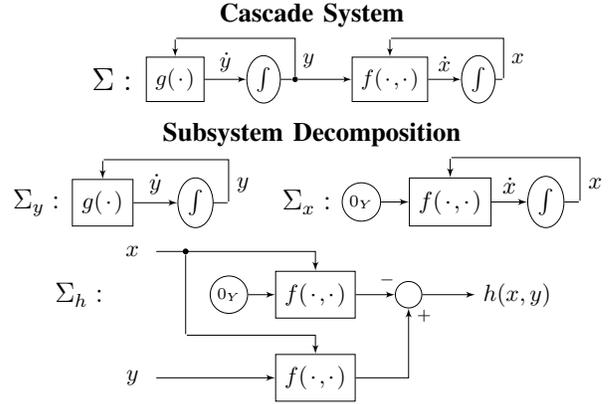
\begin{figure}
	\tikzstyle{block} = [draw, rectangle, 
	minimum height=3em, minimum width=6em]
	\tikzstyle{sum} = [draw, circle, node distance=1cm]
	\tikzstyle{input} = [coordinate]
	\tikzstyle{output} = [coordinate]
	\tikzstyle{pinstyle} = [pin edge={to-,thin,black}]
	
	\centering
	
	\textbf{Cascade System} 
	
	\vspace{4pt}
	
	\resizebox{.7\columnwidth}{!}{%
	\begin{tikzpicture}[auto, node distance=2cm,>=latex']
		
		\node [block, minimum width={26pt},minimum height={20pt}] (inner) {$g(\hspace{1pt}\cdot\hspace{1pt})$};
		\node [draw, ellipse, right of=inner, node distance=1.4cm, inner sep=2pt] (inner_integral) {$\int$};
		\node [above of=inner, node distance=.65cm, inner sep=0pt, outer sep=0pt] (inner_top) {};
		\node [right of=inner_integral, node distance=.5cm, ellipse, fill=black, minimum height=2.5pt, minimum width=2.5pt, inner sep=0pt] (inner_end) {};
		
		\node [left of=inner, node distance=1cm] (system1) {\Large $\Sigma:$};
		
		\node [block, right of=inner_integral, minimum width={34pt},minimum height={20pt}] (outer) {$f(\hspace{1pt}\cdot\hspace{1pt},\hspace{-.5pt}\cdot\hspace{1pt})$};
		\node [draw, ellipse, right of=outer, node distance=1.4cm, inner sep=2pt] (outer_integral) {$\int$};
		\node [above of=outer, node distance=.65cm, inner sep=0pt, outer sep=0pt] (outer_top) {};
		\node [right of=outer_integral, node distance=.4cm, outer sep=0pt, inner sep=0pt] (outer_end) {};

		\draw [->] (inner_integral) -- (inner_end) |-  (inner_top) node[right, pos=0.2] {$y$} -- (inner);	
		\draw [->] (inner_end) -- (outer);	
		\draw [->] (inner) -- node[name=u] {$\dot{y}$} (inner_integral);
		
		\draw [->] (outer_integral) -- (outer_end) |-  (outer_top) node[right, pos=0.25] {$x$} -- (outer);	
		\draw [->] (outer) -- node[name=u] {$\dot{x}$} (outer_integral);
		
	\end{tikzpicture}
}

	\vspace{2pt}	
	\textbf{Subsystem Decomposition} 

	\vspace{4pt}	
	\resizebox{.4\columnwidth}{!}{%
	\begin{tikzpicture}[auto, node distance=2cm,>=latex']
		
		\node [block, minimum width={26pt},minimum height={20pt}] (inner) {$g(\hspace{1pt}\cdot\hspace{1pt})$};
		\node [draw, ellipse, right of=inner, node distance=1.4cm, inner sep=2pt] (inner_integral) {$\int$};
		\node [above of=inner, node distance=.65cm, inner sep=0pt, outer sep=0pt] (inner_top) {};
		\node [right of=inner_integral, node distance=.5cm, outer sep=0pt, inner sep=0pt] (inner_end) {};
		
		\node [left of=inner, node distance=1cm] (system1) {\large $\Sigma_y:$};
		
		\draw [->] (inner_integral) -- (inner_end) |-  (inner_top) node[right, pos=0.2] {$y$} -- (inner);	
		\draw [->] (inner) -- node[name=u] {$\dot{y}$} (inner_integral);
		
	\end{tikzpicture}
}
	\resizebox{.53\columnwidth}{!}{%
	\begin{tikzpicture}[auto, node distance=2cm,>=latex']
		
		\node [block, minimum width={34pt},minimum height={20pt}] (inner) {$f(\hspace{1pt}\cdot\hspace{1pt},\hspace{-.5pt}\cdot\hspace{1pt})$};
		\node [draw, ellipse, left of=inner, node distance=1.3cm, inner sep=0pt, minimum height=16pt, minimum width=16pt] (eql) {\scriptsize $0_Y$};		
		\node [draw, ellipse, right of=inner, node distance=1.4cm, inner sep=2pt] (inner_integral) {$\int$};
		\node [above of=inner, node distance=.65cm, inner sep=0pt, outer sep=0pt] (inner_top) {};
		\node [right of=inner_integral, node distance=.5cm, outer sep=0pt, inner sep=0pt] (inner_end) {};
		
		\node [left of=eql, node distance=.8cm] (system1) {\large $\Sigma_x:$};
		
		\draw [->] (inner_integral) -- (inner_end) |-  (inner_top) node[right, pos=0.2] {$x$} -- (inner);	
		\draw [->] (inner) -- node[name=u] {$\dot{x}$} (inner_integral);
		\draw [->] (eql) -- (inner);	
	\end{tikzpicture}
}

	\resizebox{.8\columnwidth}{!}{%
	\begin{tikzpicture}[auto, node distance=2cm,>=latex']
		
		\node [block, minimum width={34pt},minimum height={20pt}] (inner) {$f(\hspace{1pt}\cdot\hspace{1pt},\hspace{-.5pt}\cdot\hspace{1pt})$};
		
		\node [block, minimum width={34pt},minimum height={20pt}, below of=inner, node distance=1.25cm] (zero) {$f(\hspace{1pt}\cdot\hspace{1pt},\hspace{-.5pt}\cdot\hspace{1pt})$};
		\node [left of=zero, node distance=2.75cm, inner sep=1pt, minimum height=20pt, minimum width=20pt] (input) {$y$};		
		
		\node [draw, ellipse, left of=inner, node distance=1.3cm, inner sep=0pt, minimum height=16pt, minimum width=16pt] (eql) {\scriptsize $0_Y$};		
		\node [sum, right of=inner, node distance=1.4cm, inner sep=4pt] (inner_integral) {};
		\node [above of=input, node distance=1.9cm,  inner sep=1pt, minimum height=20pt, minimum width=20pt] (inner_top) {$x$};
		\node [right of=inner_top, node distance=.8cm, ellipse, fill=black, minimum height=2.5pt, minimum width=2.5pt, inner sep=0pt] (inner_button) {};
		\node [below of=inner_button, node distance=1.26cm, inner sep=0pt, outer sep=0pt] (input_corner) {};
		
		\node [right of=inner_integral, node distance=.5cm, outer sep=0pt, inner sep=0pt] (inner_end) {};
		
		\node [left of=eql, node distance=2.25cm] (system1) {\large $\Sigma_h:$};
		
		\draw [->] (inner_top) -| (inner);	
		\draw [->] (inner) -- node[name=u, pos=.8] {\scriptsize $-$} (inner_integral);
		\draw [->] (inner_integral) -- node[pos=1, right] {$h(x,y)$} +(1,0);
		\draw [->] (eql) -- (inner);	
		\draw [->] (input)  -- (zero);	
		\draw [->] (zero) -| node[name=u, pos=.95, right] {\scriptsize $+$}  (inner_integral);	
		\draw [->] (inner_button) -- (input_corner) -| (zero);	
	\end{tikzpicture}
}	
\vspace{-6pt}	
	\caption{We 
		give sufficient conditions for the almost global asymptotic stability of 
		a cascade system $\Sigma$
		in terms of qualitative properties of the ``inner subsystem'' $\Sigma_y$  and the ``unforced outer subsystem'' $\Sigma_x$, as well as growth rate criteria on the ``interconnection term''
		 $\Sigma_h$ and a Lyapunov function for $\Sigma_x$. In the diagram above, $0_Y$ is the stable equilibrium of $\Sigma_y$.
}
	\label{block_diagrams}
	\vspace{-20pt}
\end{figure}

\subsection{Prior Work on Cascade Stability}

Approaches to stability certification for nonlinear cascades 
have exploited a range of structural features. Singular perturbation techniques 
\cite{Sreenath2013} 
assume a time scale separation between the ``fast'' inner subsystem and the ``slow'' outer subsystem, and show that a system's behavior tends toward that of a ``reduced'' system as the ratio between convergence rates tends to zero.
However, this approach necessitates rapid inner subsystem convergence, which can be challenging to achieve for a control system with realistic input limitations.
Another popular approach relies on the robustness of the outer subsystem to disturbances, leveraging the property of \textit{input to state stability}, which roughly requires the asymptotic response of the system under a disturbance input to be bounded 
by the size of the input (and therefore also implies global asymptotic stability of the system in the absence of disturbances). A classic result then establishes the global asymptotic stability of a cascade for which the outer subsystem is input to state stable and the inner subsystem is globally asymptotically stable \cite{Sontag1989}.
Methods avoiding robustness or time scale assumptions have relied on local exponential stability of the inner subsystem as well as growth restrictions on the ``interconnection term'' and a suitable Lyapunov function for the unforced outer subsystem to certify asymptotic stability of cascades with globally asymptotically stable subsystems \cite{Sepulchre1997}.

These global results have important applications, but their utility in geometric control is rather limited. This is because any manifold 
admitting a smooth globally asymptotically stable vector field is 
diffeomorphic to $\mathbb{R}^n$ \cite{Wilson1967}, while the state space of many physical systems (e.g. free-flying robots) is not \cite{Bhat2000}.
The most one can achieve in the smooth non-Euclidean setting, for either the subsystems or the cascade, is \textit{almost} global asymptotic stability. 
This has motivated an almost global notion of input to state stability \cite{Angeli2004}, in which an asymptotic gain holds for all but a measure zero set of initial conditions; a cascade is then guaranteed to be almost globally asymptotically stable if its outer subsystem is almost globally input to state stable and its inner subsystem is almost globally asymptotically stable. While verifying almost global input to state stability can be challenging, this can be done under conditions on the exponential instability of other equilibria as well as the ``ultimate boundedness'' of trajectories of the system even when subjected to arbitrary disturbances \cite{Angeli2010}. 

However, not all almost globally asymptotically stable cascades have an outer subsystem enjoying this property; indeed, almost global input to state stability seems to be an inherently stricter property than necessary, since it characterizes the response of the system to arbitrary disturbances, while for our purposes, the outer subsystem is almost always subjected to a converging disturbance. Yet, the lack of a comprehensive understanding of such systems has required bespoke stability certificates for almost globally asymptotically stable cascaded controllers in practice, inhibiting generalization; for example, a Lyapunov function for the combined system may be handcrafted via human intuition, even though the cascaded structure inspired the control design \cite{Lee2010}.

\subsection{Motivating Example}

\begin{figure}[t]
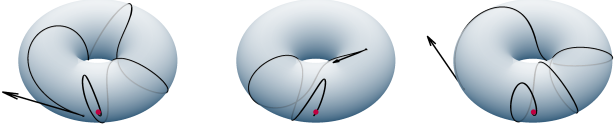

	\centering
	
	\begin{tikzpicture}[outer sep=0, inner sep=0]
		\node (image) at (0,0) { \includegraphics[width=.33\columnwidth,trim={15cm 6cm 13.5cm 4.5cm},clip] {trajectories/%
trajectory_6.0626_3.4239_6.0984_-4.5717.png%
		}};
		\node  at (0.037,-.75) [ellipse, fill=purple, minimum height=2pt, minimum width=2pt, inner sep=0pt] (inner_end) {};
	\end{tikzpicture}%
	\hfill
	\begin{tikzpicture}[outer sep=0, inner sep=0]
		\node (image) at (0,0) { \includegraphics[width=.33\columnwidth,trim={15cm 6cm 13.5cm 4.5cm},clip] {trajectories/%
trajectory_1.4482_3.4431_1.2237_-2.7408.png%
		}};
		\node  at (0.037,-.75) [ellipse, fill=purple, minimum height=2pt, minimum width=2pt, inner sep=0pt] (inner_end) {};
	\end{tikzpicture}%
	\hfill
\begin{tikzpicture}[outer sep=0, inner sep=0]
	\node (image) at (0,0) { \includegraphics[width=.33\columnwidth,trim={15cm 6cm 13.5cm 4.5cm},clip] {trajectories/%
			trajectory_5.5617_4.1329_5.0026_-4.0129.png%
	}};
	\node  at (0.037,-.75) [ellipse, fill=purple, minimum height=2pt, minimum width=2pt, inner sep=0pt] (inner_end) {};
\end{tikzpicture}%
		\vspace{-6pt}	
	\caption{A sampling of initial conditions and resulting trajectories of the motivating example system \eqref{outer_example}-\eqref{inner_example}, projected down to $\mathbb{T}^2$ from the full state space ${T\mathbb{T}^2 = T\mathbb{S}^1 \times T\mathbb{S}^1}$, where the ``small'' axis of the torus corresponds to $\theta$ and the ``large'' axis corresponds to $\phi$. All sampled trajectories converge to ${(0,0,0,0) \in T\mathbb{T}^2}$, marked in red. Despite the highly non-local and topologically complex behavior of the trajectories, our results certify the almost global asymptotic stability of the system, without the need to guess an explicit Lyapunov function for the full cascade.
 	}	
	\label{trajectories}
	\vspace{-16pt}
\end{figure}

To motivate our investigation, we present a simple representative example. Consider a cascade of the form \eqref{outer_general}-\eqref{inner_general} evolving on the tangent bundle of $\mathbb{T}^2$, given by
\begin{subequations}
	\begin{align}
	\ddot{\theta} &= -(\sin \theta + \dot{\theta})\cos 2 \phi, \label{outer_example}
	\\
	\ddot{\phi} &= -(\sin \phi + \dot{\phi}), \label{inner_example}
\end{align}
\end{subequations}
where ${x = (\theta,\dot{\theta}) \in T\mathbb{S}^1}$,
${y = (\phi,\dot{\phi}) \in T\mathbb{S}^1}$, and we make the identification ${T\mathbb{S}^1 \cong \mathbb{S}^1 \times \mathbb{R}}$ for notational convenience.
A sampling of system trajectories is shown in Fig. \ref{trajectories}.
In fact, \eqref{inner_example} describes a damped pendulum with total energy given by
${W : (\phi,\dot{\phi}) \mapsto 1 - \cos \phi + \tfrac{1}{2}\dot{\phi}^2}$, and using $W$ as a LaSalle function
it can be shown that  that ${(\phi,\dot{\phi})=(0,0)}$ is almost globally asymptotically stable for the inner subsystem \eqref{inner_example}. By the same reasoning, ${(\theta,\dot{\theta})=(0,0)}$ is also almost globally asymptotically stable for the restriction of the outer subsystem \eqref{outer_example} to the stable equilibrium of the inner subsystem.

It turns out that the entire cascade \eqref{outer_example}-\eqref{inner_example} is almost globally asymptotically stable, but the system does not satisfy the hypotheses of any of the previously discussed results. In particular, the subsystems are not globally asymptotically stable, nor is there a time scale separation between the loops. Furthermore, viewing $(\phi,\dot{\phi})$ as a disturbance to \eqref{outer_example}, it can be seen that the outer subsystem is \textit{not} almost globally input to state stable \cite[Def. 2.1]{Angeli2004}, since the response to the bounded disturbance ${(\phi,\dot{\phi})=(\pi/2,0)}$ grows unbounded from almost all initial conditions.
 Nonetheless, the results of this paper will guarantee the almost global asymptotic stability of a class of systems that includes \eqref{outer_example}-\eqref{inner_example}.

In what follows, we consider cascade systems of the form $\Sigma$ in Fig. 1, and we use properties of the decomposed subsystems shown to certify its almost global asymptotic stability. 
In Sec. \ref{sec:main_results}, we present the main results, which show that when $\Sigma_x$ and $\Sigma_y$ are almost globally asymptotically stable and locally exponentially stable and all chain recurrent points of $\Sigma_x$ are hyperbolic equilibria, then the cascade is almost globally asymptotically stable and locally exponentially stable if trajectories are bounded in forward time.
We extend this result inductively to upper triangular systems with arbitrarily many subsystems.
In Sec. \ref{sec:hypotheses}, we examine the hypotheses of the main results in greater detail,
discussing some important classes of systems enjoying the stated chain recurrence property and also showing that forward boundedness 
can be verified via growth rate criteria on $\Sigma_h$ and a Lyapunov function for $\Sigma_x$. In Sec. \ref{sec:application}, we apply our results to the motivating example, before discussing the results and concluding the paper in Secs. \ref{sec:discussion} and \ref{sec:conclusion}.

\section{Almost Global Asymptotic Stability}

\label{sec:main_results}

We begin with a brief review of relevant concepts from dynamical systems theory, adopting
the definitions of \cite{Mischaikow1995}, whose perspective on the behavior of asymptotically autonomous semiflows is a central ingredient of our approach.

\begin{definition}
	A \textit{nonautonomous semiflow} on a smooth Riemannian manifold $(M,\mu)$ is a continuous map
	\begin{equation}
		{\Phi : \{(t,s) : 0 \leq s \leq t < \infty \} \times M \rightarrow M}
	\end{equation}
	such that ${\Phi(s,s,x) = x}$ and ${\Phi\big(t,s,\Phi(s,r,x)\big) = \Phi(t,r,x)}$ for all ${ t \geq s \geq r > 0}$. A semiflow is called \textit{autonomous} when additionally,
	${\Phi(t+r,s+r,x) = \Phi(t,s,x)}$ for all $r > 0$.
\end{definition}
In the previous, the parameters $s$ and $t$ can be thought of as respective ``start'' and ``end'' times. Hereafter, we will use the shorthands
${\Phi^{(t,s)} : 
x \mapsto \Phi(t,s,x)}$
for nonautonomous semiflows and 
${	\Phi^{t} :
 x \mapsto \Phi(t,0,x)}$
for autonomous semiflows. Note that nonautonomous semiflows can be generated by ``time-varying'' vector fields, while autonomous semiflows can be generated by ``time-invariant'' vector fields.

\begin{definition}
	A nonautonomous semiflow $\Phi$ is \textit{asymptotically autonomous with limit semiflow $\Theta$} if, for any sequences 
	${t_j \rightarrow t}$,  ${s_j \rightarrow \infty}$, and ${x_j \rightarrow x}$, it holds that 
	\begin{align}		
		\Phi^{(t_j + s_j,s_j)} (x_j)  \to \Theta^t(x) \textrm{ as } j \rightarrow \infty.
	\end{align}
\end{definition}

\begin{definition}
	The \textit{equilibrium set} of an autonomous semiflow $\Phi$ is given by ${\mathcal{E}(\Phi) = \left\{x \in M : \Phi^t(x) = x \ \forall \ t \geq 0\right\}}$.
	An equilibrium  
		${0_M \in \mathcal{E}(\Phi)}$ 	
	is \textit{hyperbolic} if its linearization has no purely imaginary eigenvalues 
	\cite[p. 149]{Robinson1998}, and it
is \textit{almost globally asymptotically stable} if it is (locally) asymptotically stable and its \textit{basin of attraction}, denoted ${\mathcal{B}({0_M}) \subseteq M}$, is full measure and residual, i.e. the complement of $\mathcal{B}({0_M})$ is measure zero \cite[p.128]{SmoothLee} and meager 
	 (a countable union of nowhere dense sets) in $M$.
\end{definition}

\begin{definition}
	For an autonomous semiflow $\Phi$ on $(M,\mu)$ and constants ${\varepsilon, T > 0}$,
	an \textit{$(\varepsilon,T)$-chain} is a pair of finite sequences $(x_0, x_1,\ldots,x_n)$  and $(t_1,t_2,\ldots,t_n)$ satisfying
	\begin{equation}
		{\dist\big(\Phi^{t_i}(x_{i-1}),x_i\big) < \varepsilon}\textrm{ and }{t_i > T}, \ i = 1, 2, \ldots, n,
	\end{equation}
	where ${\dist : M \times M \rightarrow \mathbb{R}}$ is the distance induced by $\mu$. 
	The \textit{chain recurrent set} $\mathcal{R}(\Phi)$ comprises all points $x$ for which an $(\varepsilon,T)$-chain with ${x = x_0 = x_n}$ exists for every ${\varepsilon, T > 0}$.
\end{definition}

The chain recurrent set plays a central role in the topological structure of a dynamical system \cite{Robinson1998}. Our present interest in chain recurrence will revolve around the following fact.

\begin{fact}[Mischaikow, Smith, and Thieme \cite{Mischaikow1995}]
		\label{misch}
	Precompact forward trajectories of an asymptotically autonomous semiflow converge to the chain recurrent set of the limit semiflow.  
\end{fact}

\begin{remark}
	\label{choice_of_setting}
	We define chain recurrence using $(\varepsilon,T)$-chains with respect to a distance function and some $\varepsilon > 0$
	because we rely on the results of \cite{Mischaikow1995}, in which the same choice is made. A Riemannian distance function is appropriate, since then an asymptotically stable equilibrium is locally exponentially stable if and only if it is hyperbolic. In a complete Riemannian manifold (our chosen setting), a set is precompact (i.e. has compact closure) if and only if it is bounded, a more familiar notion in the control community, which will also play a convenient role in our second theorem.
\end{remark}

\subsection{Main Results}

\begin{theorem}[Almost Global Asymptotic Stability of a \mbox{Cascade}]
	\label{almost_global_stability}
	Consider a cascade on $X \times Y$ given by
	\begin{subequations}
		\begin{align}
			\dot{x} &= f(x,y), \label{outer_theorem} \\
			\dot{y} &= g(y), \label{inner_theorem} 
		\end{align}
	\end{subequations}
where 
\eqref{inner_theorem} and the unforced outer subsystem
\begin{equation}
	{\dot{x} = f(x,0_Y)}
	\label{unforced}
\end{equation}
are almost globally asymptotically stable with respect to ${0_Y \in Y}$ and ${0_X \in X}$ respectively, and moreover $0_Y$ and all chain recurrent points of \eqref{unforced} are hyperbolic equilibria.
Then, \eqref{outer_theorem}-\eqref{inner_theorem} is almost globally asymptotically stable and locally exponentially stable with respect to $(0_X,0_Y)$ as long as any forward trajectory starting in $X \times \mathcal{B}(0_Y)$ is bounded.
\end{theorem}

\begin{proof}

	Since ${X \times \mathcal{B}(0_Y)}$ is invariant for \eqref{outer_theorem}-\eqref{inner_theorem} and all forward trajectories beginning in ${X \times \mathcal{B}(0_Y)}$ are bounded, the cascade induces an autonomous semiflow
	\begin{equation}
		\Psi^t : X \times \mathcal{B}(0_Y) \rightarrow X \times \mathcal{B}(0_Y).
	\end{equation}
	Similarly, \eqref{unforced} and \eqref{inner_theorem} induce the autonomous semiflows 
	\begin{subequations}
			\begin{align}
			\Theta^t : X \rightarrow X, \ x_0 & \mapsto \mathrm{pr}_1 \circ \Psi^t(x_0,0_Y), \label{limit_semiflow} \\
			\Upsilon^t : \mathcal{B}(0_Y) \rightarrow \mathcal{B}(0_Y), \ y_0 & \mapsto \mathrm{pr}_2 \circ \Psi^t(0_X,y_0), \label{inner_semiflow}
		\end{align}
	\end{subequations}	
	where $\mathrm{pr}_1$ and $\mathrm{pr}_2$ are the natural projections onto $X$ and $Y$, and
	we have carefully chosen the domains of the semiflows.
	We observe that for each initial condition ${y_0 \in \mathcal{B}(0_Y)}$,
	\eqref{outer_theorem} may be interpreted as time-varying dynamics on $X$ given by 
	\begin{equation}
		\dot{x} = 
		 f\big(x,\Upsilon^t(y_0)\big).
	\end{equation}
	In this manner, each initial condition ${y_0 \in \mathcal{B}(0_Y)}$ induces a corresponding \textit{nonautonomous} semiflow on $X$ given by
	\begin{equation}
		\Phi^{(t,s)}_{y_0} : X \rightarrow X, \ x_0 \mapsto \textrm{pr}_1 \circ \Psi^{t-s} \big(x_0,\Upsilon^s(y_0)\big),
		\label{nonautonomous_semiflow}
	\end{equation}
	such that we may also conclude
	\begin{equation}
		\Psi^t(x_0,y_0) = \big(\Phi^{(t,0)}_{y_0}(x_0), \Upsilon^t(y_0) \big).
		\label{flow_decomposition}
	\end{equation}
With these constructions, we prove the claim in five steps.

		\begin{step}
			${\mathcal{E}(\Psi)
				 = 
				\mathcal{R}(\Theta) \times \{0_Y\}}$,
			and all points in this set are hyperbolic equilibria, of which only ${(0_X,0_Y)}$ is stable.
		\end{step}
		\begin{proof}
		\renewcommand{\qedsymbol}{$\blacktriangledown$}
			All equilibria  ${(x,y) \in {X \times \mathcal{B}(0_Y)}}$ must have ${y = 0_Y}$ by the definition of $\mathcal{B}(0_Y)$ as a basin of attraction, and therefore ${f(x,0_Y) = 0}$ i.e. $x$ is an equilibrium of \eqref{unforced}. The equality then follows from the assumption
			that ${\mathcal{R}(\Theta) \subseteq \mathcal{E}(\Theta)}$, since equilibria are always chain recurrent i.e. ${ \mathcal{E}(\Theta) \subseteq  \mathcal{R}(\Theta)}$.
			Denoting 
			the vector field on ${X \times Y}$ describing the full cascade \eqref{outer_theorem}-\eqref{inner_theorem} by ${F : (x,y) \mapsto \big(f(x,y),g(y)\big)}$, we may express its linearization at any equilibrium ${(x,0_Y) \in X \times \mathcal{B}(0_Y)}$ 
			as
			\begin{equation}
				\left.d F \right|_{(x,0_Y)} = \begin{bmatrix}
					\left. \partial_x f \,  \right|_{(x,0_Y)} &
					\left. \partial_y f \,  \right|_{(x,0_Y)} \\
					0 &
					\left.\partial_y g \, \right|_{\, 0_Y\hphantom{x,()}} \\
				\end{bmatrix}{.}
			\end{equation}
			Since the eigenvalues of a triangular block matrix are simply the eigenvalues of the blocks on the diagonal, the 
			hyperbolicity claim follows directly from
			the	hyperbolicity of $0_Y$ for \eqref{inner_theorem} and the hyperbolicity of all equilibria of \eqref{unforced}.

			An almost globally asymptotically stable system has only one stable equilibrium (since all other equilibria lie on the boundary of its basin of attraction). Therefore,
	at ${x = 0_X}$ all eigenvalues of the top left block have negative real part, but one or more eigenvalues at all other equilibria of \eqref{unforced} have positive real part. Hence $(0_X,0_Y)$ is locally exponentially stable and all other equilibria in $X \times \mathcal{B}(0_Y)$ are unstable.
		\end{proof}
		\vspace{-4pt}
		\noindent To complete the proof, it will therefore suffice to show that the stable equilibrium $(0_X,0_Y)$ is almost globally attractive.

	\begin{step}
	For any ${y_0 \in \mathcal{B}(0_Y)}$, the nonautonomous semiflow $\Phi
	_{y_0}$ is asymptotically autonomous with limit semiflow $\Theta$. 
\end{step}
\begin{proof}
		\renewcommand{\qedsymbol}{$\blacktriangledown$}
For any sequences ${t_j \rightarrow t}$,  ${s_j \rightarrow \infty}$, and ${x_j \rightarrow x}$, 
	\begin{align}
		&\begin{aligned}
					\lim_{j \rightarrow \infty} &
			\Phi^{(t_j + s_j,s_j)}_{y_0} (x_j) \\
			&= \lim_{j \rightarrow \infty} \textrm{pr}_1 \circ \Psi^{t_j} \big(x_j,\Upsilon^{s_j}(y_0)\big)
			\label{defn_nonautonomous}
		\end{aligned}
		\\
		&		\phantom{\lim_{j \rightarrow \infty}}
		= \textrm{pr}_1 \circ \Psi^{\lim\limits_{j \rightarrow \infty}  t_j} \left(\lim\limits_{j \rightarrow \infty} x_j,\lim\limits_{j \rightarrow \infty} \Upsilon^{s_j}(y_0)\right) \label{continuity_limit}\\
		&		\phantom{\lim_{j \rightarrow \infty}}
		= 
		\label{solved_limits} 
		\textrm{pr}_1 \circ \Psi^{t} (x,0_Y)
		=
		\Theta^{t} (x),
	\end{align}
where \eqref{defn_nonautonomous} follows immediately from \eqref{nonautonomous_semiflow}, \eqref{continuity_limit} is obtained by the continuity of $\mathrm{pr}_1$ and $\Psi$, and \eqref{solved_limits} relies on the attractivity of $0_Y$.
Thus for any ${y_0 \in \mathcal{B}(0_Y)}$, by definition
$\Phi_{y_0}$ is asymptotically autonomous with limit semiflow $\Theta$.
\end{proof}
		
		\begin{step}
			Every trajectory of 
			$\Psi$  
			converges to a hyperbolic equilibrium.
			
		\end{step}
		\begin{proof}
		\renewcommand{\qedsymbol}{$\blacktriangledown$}
			Together, Step 2, Fact \ref{misch}, and the boundedness (hence, precompactness) of forward trajectories of $\Psi$ imply that for each ${y_0 \in \mathcal{B}(0_Y)}$, every trajectory of $\Phi_{y_0}$ converges to $\mathcal{R}(\Theta)$, and asymptotic stability of \eqref{inner_theorem} ensures that every trajectory of $\Upsilon$ 
			converges to $0_Y$. Thus, in view of \eqref{flow_decomposition} it is clear that every trajectory of $\Psi$ 
			converges to ${\mathcal{R}(\Theta) \times \{0_Y\}}$, and all points in this set are hyperbolic equilibria by Step 1. Since hyperbolic equilibria are isolated, 
			 by continuity every trajectory
			 converges to a particular hyperbolic equilibrium.
		\end{proof}
		\begin{step}
	Almost no trajectories of \eqref{outer_theorem}-\eqref{inner_theorem} converge to an unstable equilibrium.
		\end{step}
		\begin{proof}
		\renewcommand{\qedsymbol}{$\blacktriangledown$}
			All points converging to a hyperbolic equilibrium lie on its global stable manifold, which (for an unstable equilibrium) is the union of countably many embedded submanifolds of positive codimension (see \cite[p. 73]{Palis1982} or \cite[Sec. 2.1]{Eldering2018}). Hence, this is a meager set of measure zero. Also, all unstable equilibria in ${X \times \mathcal{B}(0_Y)}$ are hyperbolic by Step 1, and there are countably many of these equilibria due to the isolation of hyperbolic equilibria and the second countability of ${X \times \mathcal{B}(0_Y)}$ \cite[Thm 2.50 and Prop. 3.11]{TopologicalLee}. Thus, the set of points in ${X \times \mathcal{B}(0_Y)}$ converging to an unstable equilibrium is a countable union of meager sets of measure zero and is thus meager (essentially by definition) and measure zero \cite[p. 128]{SmoothLee} in $X \times Y$. 
		\end{proof}
		\begin{step}
			Almost every trajectory of \eqref{outer_theorem}-\eqref{inner_theorem} converges to the stable equilibrium ${(0_X,0_Y)}$.
		\end{step}
		\begin{proof}
			Since $\mathcal{B}(0_Y)$ is full measure and residual in $Y$ by assumption, ${X \times \mathcal{B}(0_Y)}$ is full measure and residual in ${X \times Y}$. By Step 3, every initial condition in this set converges to a hyperbolic equilibrium, and by Step 4, the subset converging to an unstable equilibrium is meager and measure zero in $X \times Y$. Since the difference of a residual set of full measure by a meager set of measure zero is residual and full measure, the remainder constitutes a residual set of full measure in $X \times Y$ for which all initial conditions converge to the unique stable equilibrium ${(0_X,0_Y)}$, completing the proof.
		\end{proof}
		\renewcommand{\qedsymbol}{}
		\vspace{-20pt}
	\end{proof}

\begin{remark}
	The main potential pitfall of the unforced outer subsystem being only \textit{almost} globally asymptotically stable is the possibility of ``funneling'' a non-negligible (i.e. non-meager or  positive measure) set to a point $(x,0_Y)$, where $x$ is an unstable equilibrium of \eqref{unforced}. However, such behavior is precluded by the hyperbolicity of all unstable equilibria of \eqref{unforced}.
	This can be relaxed to the requirement that all unstable equilibria of \eqref{unforced} are isolated and have at least one eigenvalue with positive real part, similar to \cite{Angeli2010}. Then, the argument proceeds similarly, but relies on the center-stable manifold theorem instead of the stable manifold theorem.  Similarly, the hyperbolicity assumption on $0_X$ can be relaxed at the cost of local exponential stability. 
	We present the more succinct but less general result for clarity and brevity. 
	\end{remark}

\begin{corollary}[Upper Triangular System]
	\label{upper_triangular_stability}
	Consider an upper triangular system on ${X_1 \times X_2 \times \cdots \times X_n}$ given by
	\begin{subequations}
	\begin{align}
		\dot{x}_1 = f_1(x_1,x_2,\ldots,x_n)&, \label{outermost} \\
		\dot{x}_2 = f_2(x_2,\ldots,x_n)&, \label{second_system} \\[-3pt]
		 \ddots \quad \quad \quad \quad   & \nonumber \\[-3pt]
		\dot{x}_n = f_n(x_n)&, \label{innermost}
	\end{align}
	\end{subequations}
where for all ${i = 1, 2, \ldots, n}$, the unforced system
\begin{equation}
			\dot{x}_i = f_i(x_i,0_{i+1},0_{i+2},\ldots,0_n)
			\label{diagonal_system}
\end{equation}
is almost globally asymptotically stable with respect to ${0_i \in X_i}$ and its chain recurrent set contains only
hyperbolic equilibria. Then, the upper triangular system \eqref{outermost}-\eqref{innermost} is almost globally asymptotically stable and locally exponentially stable with respect to ${(0_1,0_2,\ldots,0_n)}$
 as long as
all forward trajectories of \eqref{outermost}-\eqref{innermost}
are bounded.
\end{corollary}

\begin{proof}
	The claim follows by induction. In particular, the claim is trivial for ${n=1}$, and 
	assuming it holds for ${n=k-1}$, the claim for ${n=k}$ follows by Theorem \ref{almost_global_stability} with \eqref{outermost} as the outer subsystem \eqref{outer_theorem} and \eqref{second_system}-\eqref{innermost} as the inner subsystem \eqref{inner_theorem}, i.e. ${x = x_1}$ and ${y = (x_2,x_3,\ldots,x_n)}$.
\end{proof}

\section{Hypotheses of the Main Results}

\label{sec:hypotheses}

We now explore the hypotheses of Theorem \ref{almost_global_stability} and Corollary \ref{upper_triangular_stability} in greater detail, showing how they can be verified.

\subsection{Gradient-Like Systems}

Systems with no chain recurrent points besides equilibria are often called ``gradient-like'', and the following fact shows that such a property is often easily verified.
We cannot include the proof due to space, but very similar notions are discussed in \cite[Sec. IV]{Angeli2010}, \cite[Cor. 2.4]{Benaim1995}, and \cite[Sec. 7.12]{Robinson1998}.

\begin{fact}
		\label{thm:equilibria}
		If $\mathcal{E}(\Phi)$ consists of isolated points and there is a proper\footnote{A function ${V : M \rightarrow \mathbb{R}}$ is \textit{proper} if it has compact sublevel sets, which morally generalizes the notion of ``radially unbounded'' functions on $\mathbb{R}^n$.}, continuous function ${V:M\to \R}$ 
		that is
		decreasing\footnote{A function ${f : \mathbb{R} \rightarrow \mathbb{R}}$ is \textit{decreasing} if ${f(t_2) < f(t_1)}$ whenever ${t_1 < t_2 }$. Note that this does \textit{not} imply $\dot{f}(t) < 0$ for all $t$, e.g. $t \mapsto -t^3$.} along nonequilibrium trajectories,
		then ${\mathcal{R}(\Phi) = \mathcal{E}(\Phi)}$.
\end{fact}

\begin{remark}
	From Fact \ref{thm:equilibria}, it is clear that Theorem 1 also holds if the assumption that \eqref{unforced} is almost globally asymptotically stable and gradient-like is replaced by the existence of a Lyapunov function for \eqref{unforced} around $0_X$ which is decreasing along all nonequilibrium trajectories.
	Some authors \cite{Benaim1995, Robinson1998} call this a \textit{strict} Lyapunov function, but the control community tends to reserve this term for Lyapunov functions with strictly negative derivative along nonequilibrium trajectories
		\cite{Santibanez1997}.
\end{remark}

\begin{remark}
	\label{summarize_chain_recurrence_section}
	Two important classes of systems to which Fact \ref{thm:equilibria} applies are as follows. 
	It can be shown that for a Riemannian manifold ${(Q,\kappa)}$,
	a strict Rayleigh dissipation $\nu$, and a proper Morse function ${V : Q \rightarrow \left[0,\infty\right)}$ 
	with a unique minimum at 
	${0_Q \in Q}$,
	both the gradient dynamics on $Q$ given by
				\begin{equation}
		\dot{q} = - \mathrm{grad}_\kappa\, V(q) 
		\label{gradient_descent}
	\end{equation}
	and Euler-Lagrange dynamics on $TQ$ given by\footnote{The maps ${\kappa^\flat, \nu^\flat : TQ \rightarrow T^*Q}$ and ${\kappa^\sharp, \nu^\sharp : T^*Q \rightarrow TQ}$ are the \textit{musical isomorphisms} with respect to the Riemannian metrics $\mu$ and $\nu$ \cite{BulloAndLewis2004}.}
	\begin{equation}
		\overset{\kappa}{\nabla}_{\dot{q}} \dot{q} = - \mathrm{grad}_\kappa\, V(q) - \kappa^\sharp \circ \nu^\flat (\dot{q}) 
		\label{euler_lagrange}
	\end{equation}
	are almost globally asymptotically stable and locally exponentially stable around ${0_Q \in Q}$ and ${0_{TQ} = (0_Q,0) \in TQ}$ respectively, and moreover all chain recurrent points of both systems are hyperbolic equilibria. An early, influential analysis of the previous stability properties is \cite{Koditschek1989}, while a detailed modern treatment can be found in \cite[Chap. 6]{BulloAndLewis2004}. The chain recurrence claim is immediate by Fact \ref{thm:equilibria} and the fact that the ``potential'' $V$ and the ``total energy''
	\begin{equation}
		\label{total_energy}
		W : (q,\dot{q}) \mapsto V(q) + \tfrac{1}{2}\kappa(\dot{q},\dot{q}),
	\end{equation}
	are respectively decreasing along all nonequilibrium trajectories of \eqref{gradient_descent} (cf. gradient descent) and \eqref{euler_lagrange} \cite[Prop. 4.66]{BulloAndLewis2004}. 
\end{remark}

\subsection{Boundedness of Forward Trajectories}

Forward boundedness is guaranteed when the outer subsystem evolves on a compact manifold. To use Theorem \ref{almost_global_stability} to certify the stability of a cascade evolving on a noncompact manifold (e.g. $\mathbb{R}^n$ or any tangent bundle), we require compositional criteria for forward boundedness. 
In this section, we give growth rate criteria suitable for our geometric setting on the ``{interconnection term}'' $\Sigma_h$ 
and a Lyapunov function for the unforced outer subsystem $\Sigma_x$.
The result is analogous to (and inspired by) \cite[Thm. 4.7]{Sepulchre1997}, which certifies forward boundedness in $\mathbb{R}^n$ using the standard Euclidean norm. 
In a Riemannian manifold $(X,\mu)$, we use instead 
the Riemannian distance and the dual norms on each tangent and cotangent space induced by the metric. We denote both norms by ${||\hspace{-1.5pt}\cdot\hspace{-1.5pt}||_\mu }$.

\begin{theorem}[Forward Boundedness of a Cascade]
	\label{boundedness}
	Consider a cascade on $X \times Y$ given by
	\begin{subequations}
		\begin{align}
			\dot{x} &= f(x,y), \label{outer_bound} \\
			\dot{y} &= g(y)
			.
			 \label{inner_bound} 
		\end{align}
	\end{subequations}
	 Suppose the following conditions hold on the subsystems:
\setlist[enumerate,1]{leftmargin=.75cm}
		\begin{enumerate}
			\item[$\Sigma_y\hspace{-2.5pt}:$] For \eqref{inner_bound}, ${0_Y \in Y}$ is a stable hyperbolic equilibrium. 
				\item[$\Sigma_x\hspace{-2.5pt}:$]
		${W : X \rightarrow \mathbb{R}_{\geq 0}}$ is a proper Lyapunov function  for 
		\begin{equation}
			\dot{x} = f(x,0_Y) \label{unforced_bound}
		\end{equation}
		such that for some constants ${\lambda \geq 0}, {d_0 \geq 1}$, 
		\begin{align}
			\left|\left|\mathrm{d}W_x\right|\right|_{{\mu}} \dist(0_X,x) \leq \lambda \, W(x)
			\label{outer_lyapunov_growth_restriction}
		\end{align}
		for all 
		${(x,y) \in \{x \in M : \dist(0_X,x) \geq d_0 \} \times \mathcal{B}(0_Y)}$.
				\item[$\Sigma_h\hspace{-2.5pt}:$]
For some continuous maps 
${\alpha, \, \beta : \mathcal{B}(0_Y) \hspace{-1pt} \to \hspace{-1pt} \mathbb{R}_{\geq 0}}$ that are
vanishing and differentiable at $0_Y$,
the
interconnection term
${h : (x,y) \mapsto f(x,y) - f(x,0_Y)}$
satisfies 
\begin{align}
	\left|\left|h(x,y)\right|\right|_{{\mu}} \leq \alpha(y) \dist(0_X,x) + \beta(y).
	\label{interconnection_growth_restriction}
\end{align}
			\end{enumerate}
	Then, the trajectory of \eqref{outer_bound}-\eqref{inner_bound} through any initial condition ${(x_0,y_0) \in X \times \mathcal{B}(0_Y)}$ is bounded for all forward time.
\end{theorem}

\begin{proof}
	
	Since $W$ is a proper Lyapunov function for \eqref{unforced_bound}, the forward trajectory through any initial condition of the form ${(x_0,0_Y)}$ is bounded, so it suffices to consider initial conditions $(x_0,y_0)$ with ${y_0\neq 0_Y}$. Fix ${(x_0,y_0)\in X\times \mathcal{B}(0_Y)}$ with ${y_0\neq 0_Y}$ and let $\big(x(t),y(t)\big)$ denote its forward trajectory. 
	
	\begin{step}
		\label{bound_alpha_beta}
		There exist positive constants $A$ and $\omega$ such that
		$\alpha\big(y(t)\big) + \beta\big(y(t)\big)\leq Ae^{-\omega t}$ for all $t\geq 0$.
	\end{step}
	\begin{proof}
		\renewcommand{\qedsymbol}{$\blacktriangledown$}
		Let ${d(t)\coloneqq \text{dist}(0_Y,y(t)) > 0}$.
		Since 
		$0_Y$ is hyperbolic, 
there exist $C_0, \omega > 0$ such that, for all $t\geq 0$,
		\begin{equation}\label{eq:d-bound}
			d(t)\leq C_0e^{-\omega t}.
		\end{equation}
		Next, since 
		$\alpha$ and $\beta$ are vanishing and differentiable at $0_Y$, a local coordinate calculation (using uniform equivalence of continuous Riemannian metrics over compact sets) shows
		\begin{equation}
			\limsup_{t\to\infty}\frac{\alpha\big(y(t)\big) + \beta\big(y(t)\big)}{d(t)} < \infty. 
		\end{equation}
		The quotient in the previous limit is a continuous function of $t$ and thus is bounded for all ${t \geq 0}$ by some ${C_1 > 0}$.
		With \eqref{eq:d-bound}, this yields the desired bound
		with $A:= C_0 C_1$.
	\end{proof}
	
	\begin{step}
		$W\big(x(t)\big)$ is bounded for all $t \geq 0$.
	\end{step}
	\begin{proof}
		\renewcommand{\qedsymbol}{$\blacktriangledown$}
		
		Since $W$ is a Lyapunov function for \eqref{unforced_bound}, we have
		\begin{align}
			\dot{W} 
			&\leq \mathrm{d}W_x \big( h(x,y) \big)
			\leq \left|\left|\mathrm{d}W_x \right|\right|_{{\mu}} \,
			\left|\left|h(x,y)\right|\right|_{{\mu}}
			\\&\leq \left|\left|\mathrm{d}W_x \right|\right|_{{\mu}} \,
			\big(\alpha(y) \dist(0_X,x) + \beta(y)\big),
		\end{align}
		from
		 \eqref{interconnection_growth_restriction}.
		Hence, whenever 
$\dist(0_X,x) \geq d_0 \geq 1$, we have
\begin{align}
	\dot{W}
	&\leq 
	||\mathrm{d}W_x ||_{{\mu}} \dist(0_X,x)  \big(\alpha(y) + \beta(y)\big).
\label{for_large_distances}
		\end{align}
	Define  
	$W_0 = \sup_{\{x \, : \, \dist(0_X,x) \, \leq \, d_0\}} W(x)$ and
consider any ${t_2 \geq t_1 \geq 0}$ where ${W(x([t_1,t_2]))\subseteq [ W_0,\infty)}$.
				Then for all ${t\in [t_1,t_2]}$,  \eqref{outer_lyapunov_growth_restriction}, \eqref{for_large_distances}, and the conclusion of Step 1 imply 
				\begin{equation}
						\tfrac{d}{dt}W\big(x(t)\big)\leq
						\lambda A e^{-\omega t} \, W\big(x(t)\big). \label{time_varying_W_bound} 
					\end{equation}
				By Gr\"{o}nwall's inequality, we obtain the bound
				\begin{align}
						W(x(t_2)) \leq  e^ {\int_{t_1}^{t_2} \lambda A e^{\text{-}\omega t} dt}  W(x(t_1))
						\leq e^{\frac{\lambda A}{\omega}}  W(x(t_1)).
				\end{align}
				This implies that for all $t\geq 0$,
				$$W(x(t))\leq C:= e^\frac{\lambda A}{\omega} \max\big\{W_0,W\big(x(0)\big)\big\} . \qedhere $$
	\end{proof}
	\noindent 
	Thus, since $W$ is proper and $0_Y$ is attractive,
	 it follows that 
	$\big(x(t),y(t)\big)$ is bounded for all $t \geq 0$.
\end{proof}
	A similar approach can be iterated (as in Corollary \ref{upper_triangular_stability}) to certify forward boundedness of an upper triangular system.

\section{Application of the Results}

\label{sec:application}

We now revisit the motivating example \eqref{outer_example}-\eqref{inner_example}.
It is easily verified that \eqref{inner_example} takes the form of the Euler-Lagrange dynamics \eqref{euler_lagrange} 
for the kinetic energy metric and Rayleigh dissipation ${\kappa = \nu = d \phi \otimes d\phi}$ and the Morse potential function ${V : \mathbb{S}^1 \rightarrow \mathbb{R}, \, \phi \mapsto 1-\cos \phi}$.
Thus by Remark \ref{summarize_chain_recurrence_section}, \eqref{inner_example} is almost globally asymptotically stable and locally exponentially stable with respect to ${y = (\phi,\dot{\phi}) = (0,0)}$, and moreover its chain recurrent set consists solely of hyperbolic equilibria. Clearly, the same is true for
${x = (\theta,\dot{\theta}) = (0,0)}$ with respect to the restriction of \eqref{outer_example} to ${y = (\phi,\dot{\phi}) =(0,0)}$. 

By Theorem \ref{almost_global_stability}, for almost global asymptotic stability it will suffice to show forward boundedness, which
we accomplish using 
the total energy \eqref{total_energy}. 
The natural choice of metric on the tangent bundle ${X = T\mathbb{S}^1}$ is the Sasaki metric \cite{Gudmundsson2002} i.e.
${\tilde{\kappa} = d \theta \otimes d\theta + d \dot\theta \otimes d\dot\theta}$. Then, considering (without loss of generality) the range of angles ${\theta \in [-\pi,\pi)}$, we have
\begin{align}
	\underbrace{
			\sqrt{\sin^2\theta + \dot{\theta}^2}
	}_{||\mathrm{d}W_x||_{\tilde{\kappa}}}  
	\underbrace{\sqrt{\theta^2  + \dot{\theta}^2}}_{
		\dist(0_X,x)
	}
	&\leq \theta^2 + \dot{\theta}^2 
	\leq 
	\underbrace{\tfrac{\pi^2}{2}}_\lambda 
	\underbrace{(1-\cos \theta +  \tfrac{\dot{\theta}^2}{2} )}_{W(x)}, 
	\nonumber
\end{align}
since ${|\sin \theta| \leq |\theta|}$ and ${\theta^2 \leq \frac{\pi^2}{2}(1-\cos\theta)}$ for ${\theta \in [-\pi,\pi)}$,
verifying that \eqref{outer_lyapunov_growth_restriction} holds.
Furthermore,  we compute
\begin{align}
	\underbrace{(1-\cos 2\phi)
		(\sin \theta + \dot{\theta})}_{
	|| h(x,y) ||_{\tilde{\kappa}}
	}
	&\leq 
	\underbrace{(1-\cos 2\phi) \sqrt{2}}_{\alpha(y)}
	\underbrace{\sqrt{\cramped{\theta^2 + \dot{\theta}^2}}
	}_{\dist(0_X,x)},
\end{align}
so \eqref{interconnection_growth_restriction} holds as well.
Thus it follows by Theorem \ref{boundedness} that all forward trajectories of \eqref{outer_example}-\eqref{inner_example} with ${y = (\phi,\dot{\phi})}$ starting in the basin of attraction of \eqref{inner_example} are bounded, and so the system 
is almost globally asymptotically stable and locally exponentially stable with respect to ${(0,0,0,0) \in T\mathbb{T}^2}$.

\section{Discussion}

\label{sec:discussion}

The disturbance robustness of systems with some similar properties, and the connection to gradient-like systems, was discussed in \cite[Sec. IV]{Angeli2010}. However, those results (when combined with \cite{Angeli2004}) can only certify the stability of a cascade if the outer subsystem is almost globally input to state stable, requiring also ``ultimate boundedness'' with respect to \textit{any} bounded disturbance, absent from our motivating example.

Our main results show that an upper triangular system consisting of almost globally asymptotically stable, gradient-like subsystems with no degenerate equilibria is itself almost globally asymptotically stable if all forward trajectories are bounded. Since \textit{globally} asymptotically stable systems are gradient-like (with a single chain recurrent point, i.e. the stable equilibrium), the result is analogous to the fact that a cascade of globally asymptotically stable subsystems is globally asymptotically stable as long as all trajectories are bounded \cite[Prop. 4.1]{Sepulchre1997}. Our second theorem generalizes a classic compositional method of verifying forward boundedness using growth rate criteria on the interconnection term and a Lyapunov function for the unforced outer subsystem. Unfortunately, the Riemannian distance function used in our condition can be difficult to compute explicitly in complex examples, but even loose upper and lower bounds on this distance could potentially be used to verify the inequalities.

Gradient-like dynamics are common in the closed-loop subsystems of geometric controllers \cite{Lee2010,Sreenath2013}. Indeed, since cascades of mechanical systems with suitable dissipation and potential enjoy the required  stability and chain recurrence properties, we see promising directions for the constructive synthesis of cascaded geometric controllers with almost global asymptotic stability for underactuated robotic systems, e.g. those possessing a geometric flat output (such as quadrotors and aerial manipulators) \cite{Welde2023}, which enjoy a cascade-like structure where the evolution of the system in the shape space is uniquely determined by the evolution in the symmetry group. Indeed, for a reference trajectory with constant acceleration, the error dynamics of the geometric quadrotor controller proposed in \cite{Lee2010} take the form \eqref{outer_general}-\eqref{inner_general}, and the subsystems are dissipative mechanical systems.

\section{Conclusion}

\label{sec:conclusion}

In this work, we present sufficient conditions for the almost global asymptotic stability of a cascade in which the subsystems are only almost globally asymptotically stable. The result is extended inductively to upper triangular systems of arbitrary size. The approach relies on the forward boundedness of trajectories (which can be verified by growth rate criteria on the interconnection term and on a Lyapunov function for the unforced outer subsystem) and the absence of chain recurrent points other than hyperbolic equilibria in the unforced outer subsystem. 
The results are analogous to classic results for cascades of globally asymptotically stable systems.
 The compositional nature of the criteria facilitates stability verification for arbitrarily complex cascades, so long as the subsystems enjoy certain fundamental  properties.

\bibliographystyle{IEEEtran}
\bibliography{IEEEabrv,refs}

\clearpage

\appendix

\section*{A. Chain Recurrence and Decreasing Functions}

In this appendix, we localize the chain recurrent set $\mathcal{R}(\Phi)$ of a continuous semiflow ${\Phi^t : M \rightarrow M}$ to a particular subset of the state space, with the aid of a function which is decreasing along all trajectories outside this subset. 
While similar results appear to be known \cite{Benaim1995}, we do not know of a reference providing these facts in our exact setting (e.g. for semiflows on possibly noncompact manifolds).

\begin{theorem}\label{th:R-subset-E}
	Assume there exists a proper continuous function ${V\colon M\to \R}$ and a subset ${S\subseteq M}$ such that $V(S)$ is nowhere dense in $\R$ and $V$ is decreasing on trajectories outside of $S$.
	Then ${\mathcal{R}(\Phi)\subseteq S}$.
\end{theorem}
\begin{proof}
	Fix any $x\not \in S$.
	Since $V(S)$ is nowhere dense and ${x\not  \in S}$, there exist $b>a>0$ such that 
	\begin{equation}
		\begin{gathered}
			[a,b]\subseteq V(\Phi^{[0,\infty)}(x)) \text{ and }
			\{a\leq V \leq b\}\cap S = \varnothing.
		\end{gathered}
	\end{equation}
	Since $V$ is decreasing along trajectories outside of $S$ and since $V$-sublevel sets are compact, this implies  the existence of $T>0$ such that $V(\Phi^T(x)) \leq a$ and $\Phi^{[T,\infty)}(\{V\leq b\})\subseteq \{V\leq a\}$.
	Compactness of $\{V\leq a\}$ also implies the existence of $\varepsilon > 0$ such that the distance between any point in $\{V\leq a\}$ and any point in $\{V \geq b\}$ is at least $2\varepsilon$. 
	By construction there does not exist an $(\varepsilon, T)$-chain from $x$ to $x$, so $x$ is not chain recurrent. 
	This completes the proof.
\end{proof}

\begin{corollary}\label{co:E-isolated-components}
	Assume there exists a proper continuous function ${V\colon M\to \R}$ and a subset ${S\subseteq M}$ such that $V$ is constant on each connected component of $S$, each connected component of $S$ is isolated, and  $V$ is decreasing on trajectories outside of $S$.
	Then ${\mathcal{R}(\Phi)\subseteq S}$.
\end{corollary}
\begin{proof}
	For each ${t\in \R}$, the set ${S_t\coloneqq S\cap \{V\leq t\}}$ is compact since $V$ is proper, so each $S_t$ has finitely many components since components of $S$ (hence also $S_t$) are isolated.
	This implies that ${V(S_{t+1}\setminus S_t)\subseteq (t,t+1]}$ is finite for each $t$.
	Thus, ${V(S)=\bigcup_{n=1}^{\infty}V(S_{n+1}\setminus S_n)}$ is not dense in any nonempty open subset of $\R$, so $V(S)$ is nowhere dense. 
	The desired result now follows from Theorem~\ref{th:R-subset-E}.
\end{proof}

We conclude the analysis with the derivation of Fact \ref{thm:equilibria}, which is related to the analysis in \cite[Cor. 2.4]{Benaim1995}, however the setting of that work differs from our own, since it considers stochastic processes evolving on $\mathbb{R}^n$.

\begin{corollary}\label{co:equilibria}
		If $\mathcal{E}(\Phi)$ consists of isolated points and there is a proper, continuous function ${V:M\to \R}$ 
that is
decreasing along nonequilibrium trajectories,
then ${\mathcal{R}(\Phi) = \mathcal{E}(\Phi)}$.
\end{corollary}
\begin{proof}
	Since $\mathcal{E}(\Phi) \subseteq \mathcal{R}(\Phi)$, it suffices to show that $\mathcal{R}(\Phi)\subseteq \mathcal{E}(\Phi)$.
	Since each connected component of ${S = \mathcal{E}(\Phi)}$ is a singleton, $V$ is automatically constant on each component of $S$, so the desired result follows from Corollary~\ref{co:E-isolated-components}.
\end{proof}

\begin{remark}
	If the Riemannian metric on $M$ is replaced with a new one, the conclusions of Theorem~\ref{th:R-subset-E} and Corollary~\ref{co:E-isolated-components} imply that the new chain recurrent set is still contained in $S$, and the conclusion of Corollary~\ref{co:equilibria} implies that the new chain recurrent set coincides with the old, i.e. $\mathcal{E}(\Phi)$. 
\end{remark}

\section*{B. Two Classes of Gradient-Like Systems}

		In this appendix, we consider two important classes of systems of particular relevance to geometric control design, which are already widely known to be almost globally asymptotically stable. With the aid of Corollary \ref{co:equilibria}, we verify the lesser-known fact that such systems have a chain recurrent set consisting solely of hyperbolic equilibria.
		The following facts are not particularly novel (see e.g. the discussion in \cite[Sec. IV]{Angeli2010} and \cite{Benaim1995}), but they are relevant to our larger interests, so we present them for completeness.

		\subsection{Gradient Systems}

		We first consider dynamical systems induced by descending the gradient of a \textit{Morse function} (i.e. a function whose critical points are all nondegenerate \cite{BulloAndLewis2004}) with a unique minimum. We note that Morse functions with unique minima, including ``perfect'' ones with the minimum possible number of critical points, are well-known for those manifolds typically encountered in geometric control.

		\begin{proposition}[Gradient System]
			\label{first_order_dynamics}
			For a Riemannian manifold ${(Q,\kappa)}$ and a proper Morse function ${V : Q \rightarrow \left[0,\infty\right)}$ 
			with a unique minimum at 
			${0_Q \in Q}$, the dynamical system 
			\begin{equation}
				\dot{q} = - \mathrm{grad}_\kappa\, V(q) 
				\label{gradient_descent}
			\end{equation}
			is almost globally asymptotically stable and locally exponentially stable with respect to $0_Q$, and all chain recurrent points of \eqref{gradient_descent} are hyperbolic equilibria.
		\end{proposition}
		
		\begin{proof}
			The almost global asymptotic stability of \eqref{gradient_descent} with respect to $0_Q$ is proved in \cite[Proposition 2.1]{Koditschek1989} for compact $Q$. However, the extension to the noncompact case is immediate since the sublevel sets of $V$ are compact and forward invariant, since by direct computation, 
			\begin{align}
				\dot{V} 
				&= dV(-\mathrm{grad}_\kappa V)
				= -\kappa\big(\mathrm{grad}_\kappa V, \mathrm{grad}_\kappa V \big) \leq 0.
			\end{align}
			Since the equilibria of \eqref{gradient_descent} are simply the critical points of $V$, the nondegeneracy of the critical points of Morse functions ensures hyperbolicity and therefore the local exponential stability of $0_Q$.
			Finally, since $V$ is decreasing on nonequilibrium trajectories and hyperbolic equilibria are isolated, Corollary \ref{co:equilibria} implies that the chain recurrent set of \eqref{gradient_descent} is exactly the set of equilibria.
		\end{proof}

		\subsection{Dissipative Mechanical Systems}
		
		We now turn our attention to the important class of dissipative mechanical systems arising from kinetic energy, potential energy, and damping. Such systems have been studied at length, since the introduction of artificial dissipation and potential shaping via feedback can result in closed-loop dynamics of this form with desirable limit behavior. 
		We direct the reader to the seminal work \cite{Koditschek1989} which studies the global stability properties of such systems, as well as the more recent reference \cite[Chap. 6]{BulloAndLewis2004} which provides a comprehensive and detailed overview.
		 Closed loop dynamics of this form have enabled trajectory tracking on arbitrary Lie groups 
		 and have also featured in the inner and outer subsystems of cascaded geometric controllers for underactuated robotic systems \cite{Lee2010,Sreenath2013}.

		\begin{proposition}[Dissipative Mechanical System]
			\label{second_order_dynamics}
			For a Riemannian manifold ${(Q,\kappa)}$,
			a strict Rayleigh dissipation $\nu$, and a proper Morse function ${V : Q \rightarrow \left[0,\infty\right)}$ 
			with a unique minimum at 
			${0_Q \in Q}$,
			the Euler-Lagrange dynamical system
			\begin{equation}
				\overset{\kappa}{\nabla}_{\dot{q}} \dot{q} = - \mathrm{grad}_\kappa\, V(q) - \kappa^\sharp \circ \nu^\flat (\dot{q}) 
				\label{appendix:euler_lagrange}
			\end{equation}
			is almost globally asymptotically stable and locally exponentially stable with respect to ${0_{TQ} = (0_Q,0) \in TQ}$, and all chain recurrent points of \eqref{appendix:euler_lagrange} are hyperbolic equilibria.
		\end{proposition}
		
		\begin{proof}
			It is clear that the equilibrium set of \eqref{appendix:euler_lagrange} is precisely the image of the critical points of $V$ in the zero section of $TQ$, and moreover these equilibria can be verified to be hyperbolic since $\nu$ is a strict linear dissipation and the critical points of a Morse function are nondegenerate. Moreover, only $0_{TQ}$ is (locally exponentially) stable, while all other equilibria are unstable, since $0_Q$ is the unique minimum of $V$.
			Considering the total energy function given by
			\begin{equation}
				\label{total_energy_appendix}
				W : (q,\dot{q}) \mapsto V(q) + \tfrac{1}{2}\kappa(\dot{q},\dot{q}),
			\end{equation}
			we compute
			\begin{align}
				\dot{W} 
				&= dV(q) \dot{q}  + \kappa(	\overset{\kappa}{\nabla}_{\dot{q}} \dot{q},\dot{q}) \\
				&= dV(q) \dot{q}  + \kappa(	- \mathrm{grad}_\kappa\, V(q) - \kappa^\sharp \circ \nu^\flat (\dot{q})   \, , \, \dot{q} ) \\
				&= dV(q) \dot{q}  - dV(q) \dot{q}  - \nu(\dot{q},\dot{q}) 
				= -\nu(\dot{q},\dot{q}) \leq 0.
			\end{align}	
			For any trajectory $t\mapsto q(t)$ of the Euler-Lagrange dynamics, \eqref{appendix:euler_lagrange} and strictness of $\nu$ imply that $\nu(\dot{q}(t),\dot{q}(t))>0$ for almost all $t$ if and only if the trajectory is nonequilibrium, so $W$ is decreasing along nonequilibrium trajectories. Thus, by Corollary \ref{co:equilibria}, the chain recurrent set of \eqref{appendix:euler_lagrange} is exactly the set of equilibria. Becuase $W$ is proper (since $V$ is proper and $\nu$ is positive definite) and nonincreasing along trajectories, all forward trajectories are precompact and therefore converge to the chain recurrent set \cite{Mischaikow1995}. Since hyperbolic equilibria are isolated, all trajectories converge to some equilibrium. Application of the global stable manifold theorem shows that almost no trajectories converge to an unstable hyperbolic equilibrium, so the unique stable equilibrium $0_{TQ}$ is almost globally asymptotically stable.
		\end{proof}
		
		\begin{remark}
			A primary contribution of \cite{Koditschek1989} is the observation that the global limit behavior of a dissipative mechanical system is essentially determined by the global limit behavior of the associated gradient system, which is often called the ``lifting property'' of dissipative mechanical systems \cite{BulloAndLewis2004}. Here, we have shown that a similar lifting property holds for these systems in regards to the chain recurrent set.
		\end{remark}

\end{document}